\documentclass[10pt,reqno]{amsart}
\usepackage{amsmath, amsthm, amssymb, enumerate}
\usepackage[usenames]{color}
\usepackage{eepic,epic}
\usepackage{esint}

\newtheorem{thm}{Theorem}[section]

\newtheorem{lem}[thm]{Lemma}
\newtheorem{prop}[thm]{Proposition}
\theoremstyle{definition}

\theoremstyle{remark}

\numberwithin{equation}{section}

\newcommand{\R}{\mathbb{R}}

\begin{document}

\title[]
{}


\title[Nonrelativistic limit of Pseudo Relativistic NSE]
{Nonrelativistic limit of standing waves for pseudo-relativistic nonlinear Schr\"odinger equations}

\author{Woocheol Choi}
\address[Woocheol Choi]{
School of Mathematics, Korea Institute for Advanced Study,
Hoegiro 85, Dongdaemun-gu, Seoul 130-722, Republic of Korea}
\email{wchoi@kias.re.kr}

\author{Jinmyoung Seok}
\address[Jinmyoung Seok]{Department of Mathematics, Kyonggi University,
154-42 Gwanggyosan-ro, Yeongtong-gu, Suwon 443-760, Republic of Korea}
\email{jmseok@kgu.ac.kr}

\begin{abstract}
In this paper we study standing waves for pseudo-relativistic nonlinear Schr\"odinger equations.
In the first part we find ground state solutions.
We also prove that they have one sign and are radially symmetric.
The second part is devoted to take nonrelativistic limit of the ground state solutions in $H^1 (\mathbb{R}^n)$ space.

\end{abstract}

\maketitle

\section{Introduction}
In this paper we are concerned with standing waves of a relativistic wave equation
\begin{equation}\label{eq-i1}
i \frac{\partial \psi}{\partial t} = \sqrt{(-c^2 \Delta + m^2 c^4)}\psi - mc^2 \psi +F'(|\psi|^2)\psi
\quad \text{in } \mathbb{R}^n \times \R, \quad n \geq 2
\end{equation}
where
$c > 0$ denotes the speed of light, $m > 0$ represents the particle mass and
$F : [0, \infty) \rightarrow \mathbb{R}$ is a internal potential function.
This equation is often referred to as the pseudo-relativistic nonlinear Schr\"odinger equation
because it can be considered as one of relativistic versions of usual nonlinear Schr\"odinger equation
\[
i \frac{\partial\psi}{\partial t} = -\frac{1}{2m}\Delta\psi+ F'(|\psi|^2)\psi.
\]
The equation \eqref{eq-i1} describes dynamics of systems consisting of identical spin-$0$ bosons
whose motions are relativistic (see \cite{LT, LY}), for example, boson stars.
We refer to \cite{DSS, ES, FJL1, FJL2, FJL3, L1, LT, LY} for the rigorous derivation and dynamical study.
In this work we pay attention to the standard power type nonlinearity
$F(s) = -\frac1ps^{p/2},\, p > 2$.
A more intricate case where the power type nonlinearity is combined with the Hartree type nonlinearity will be studied in a forthcoming work \cite{CS2}.
\

Inserting the standing wave ansatz $\psi(x,t) = e^{i\mu t} u(x)$ to \eqref{eq-i1},
we obtain the following nonlinear scalar field equation
\begin{equation}\label{main-eq}
 \sqrt{-c^2\Delta + m^2c^4}u -mc^2 u +\mu u = |u|^{p-2}u \quad \text{ in } \mathbb{R}^n,\quad n \geq 2,
\end{equation}
which is a relatistic verstion of the limit equation
\begin{equation}\label{limit-eq}
-\frac{1}{2m}\Delta u + \mu u = |u|^{p-2}u.
\end{equation}

This problem has been studied recently by Coti Zelati-Nolasco \cite{CN2} and Tan-Wang-Yang \cite{TWY}.
By just letting $c = 1$ (choosing units such that $c = 1$), it is shown in \cite{CN2} that
there exists a radially symmetric positive solution of \eqref{main-eq}
for all $\mu > 0$ and $p \in (2, 2n/(n-1))$.
Also, when setting $c = 1,\, \mu = m = 1$, the existence of a ground state solution for
$p \in (2, 2n/(n-1))$ and the nonexistence of bounded solution for $p \geq 2n/(n-1)$ are proved in \cite{TWY}.

In this paper, we investigate the ground state solutions of \eqref{main-eq} further.
We shall study their existence for full range of $\mu > 0$,
their qualitative properties such as symmetry or sign definiteness,
and their nonrelativistic limit.
We note the associated functional of \eqref{main-eq} is defined on $H^{1/2}(\mathbb{R}^n)$ as follows:
\begin{equation}\label{eq-E}
\begin{split}
I(u)
&= \frac{1}{2} \int_{\mathbb{R}^n}(\sqrt{c^2 \xi^2 + m^2 c^4}-mc^2+\mu)|\hat{u} (\xi)|^2\,d\xi -\frac{1}{p} \int_{\mathbb{R}^n} |u|^p\,dx,
\end{split}
\end{equation}
where $\hat{u}(\xi)$ denotes the Fourier transform $(2\pi)^{-n/2} \int_{\mathbb{R}^n} e^{-i x \cdot \xi} u(x) dx$.
Every critical point $u$ of $I$ solves \eqref{main-eq} and vice versa.
We say a critical point $u \in H^{1/2}(\mathbb{R}^n)$ of $I$ is a ground state solution of \eqref{main-eq} if
it satisfies
\[
I(u) = \min\{ I(v)  ~|~ v \in H^{1/2}(\mathbb{R}^n) \setminus \{0\},\, I'(v) = 0\}.
\]

\begin{thm}[Existence]\label{existence}
Let $m,\, \mu > 0,\, c \geq 1$ and $p \in (2, 2n/(n-1))$ be fixed.
Then, there exists a ground state solution of \eqref{main-eq}.
\end{thm}
In Theorem \ref{existence} we shall find the solution through the energy minimization problem on the Nehari manifold. This approach enables us to show that the solution has the minimal energy among all possible solutions of \eqref{main-eq}. In addition it will be convenient to prove that the uniform $L^p$ boundedness of ground state solutions $u_c$ with respect to $c \geq 1$, which is important in proving the nonrelativistic limit $c \rightarrow \infty$ of $u_c$ in Theorem \ref{nrlimit} below.

For the limit equation \eqref{limit-eq}, it was shown that the ground state solution has one sign
and radially symmetric up to a translation.
It is however not known whether the positive radial solution of \eqref{main-eq}
constructed in \cite{CN2} is a ground state solution or not. It is also not clear whether
the ground state solution of \eqref{main-eq} obtained in \cite{TWY} is radially symmetric
up to a translation or not.
We clarify this issue in the following theorem.
\begin{thm}[Qualitative properties of ground states]\label{qp}
Let $m,\, \mu > 0,\, c \geq 1$ and $p \in (2, 2n/(n-1))$ be fixed.
Suppose $u \in H^{1/2}(\mathbb{R}^n)$ is a ground state solution of \eqref{main-eq}.
Then $u$ has one sign and is radially symmetric up to a translation.
\end{thm}
\noindent We note that if $u$ is a solution of \eqref{main-eq}, then $-u$ is also 
a solution of \eqref{main-eq}. Thus we may assume ground states found in 
Theorem \ref{existence} are positive throughout the paper.

The main concern of this paper is to justify nonrelativistic limit process $c \to \infty$ 
Applying the Taylor expansion to the expression $\sqrt{c^2|\xi|^2 +m^2c^4} - mc^2$, we see
\begin{equation}\label{eq-formal}
\begin{split}
mc^2 \left( \sqrt{-\frac{|\xi|^2}{m^2 c^2} +1} -1\right) &= mc^2 \left( 1+ \frac{1}{2} \left(\frac{|\xi|^2}{m^2 c^2}\right) + O\left( \left( \frac{|\xi|^2}{m^2 c^2}\right)^2 \right) -1\right)
\\
& = \frac{1}{2m} |\xi|^2+ O\left(\left( \frac{(|\xi|^2)^2}{m^3 c^2}\right)\right)
\end{split}
\end{equation}
so that the relativistic kinetic energy $\sqrt{-c^2 \Delta +m^2 c^4} - mc^2$ formally converges to the nonrelativistic kinetic energy $-\frac{1}{2m} \Delta$. In this regard, a natural question is to ask whether or not as $c \rightarrow \infty$
an one parameter family of solutions $\{ u_c\}_{c>1}$ of nonlinear the problem \eqref{main-eq}
converges, up to a subsequence, to a solution of the problem
\begin{equation*}
-\frac{1}{2m}\Delta u + \mu u = |u|^{p-2} u.
\end{equation*}
Though this question is quite fundamental, up to authors' knowledge, there has been no rigorous result dealing with this problem under this setting.
We remark that Lenzmann \cite{L2} considered the nonrelativistic limit for
$L^2$-constrained minimization problems with the hartree type nonlinearity in dimension three. For evolution problems, Tsutsumi \cite{Ts}, Machihara-Nakanishi-Ozawa \cite{MNO}, and Masmoudi-Naknishi \cite{MN} obtained the rigorous result for the nonrelativistic derivation of the nonlinear Schr\"odinger equations from the nonlinear Klein-Gordon equations.
The final result of this paper is to show the $H^1$ convergence of
the ground state solution of \eqref{main-eq} to a solution of the limit equation \eqref{limit-eq} as $c \to \infty.$
\begin{thm}[Nonrelativistic limit]\label{nrlimit}
Let $m,\, \mu > 0$ and $p \in (2, 2n/(n-1))$ be fixed.
For given $c \geq 1$, let $u_c$ be the positive radially symmetric ground state solution of \eqref{main-eq} obtained in Theorem \ref{existence}.
Then, by choosing a subsequence, $u_c$ converges to a unique positive radially symmetric solution $u_\infty$ of
\[
-\frac{1}{2m}\Delta u_\infty + \mu u_\infty = |u_\infty|^{p-2}u_\infty
\]
as $c \to \infty$ in the sense that
\begin{equation}
\lim_{c \rightarrow \infty} \| u_c - u_{\infty}\|_{H^1(\mathbb{R}^n)} = 0.
\end{equation}
\end{thm}
It is a nontrivial issue to get the convergence in $H^{1}(\mathbb{R}^n)$ space because in the formal expansion \eqref{eq-formal} the error term $O\left(\frac{|\xi|^4}{m^3 c^2}\right)$ requires a higher regularity for taking the limit $c \rightarrow \infty$ though the natural space to find the solution $u_c$ would be $H^{1/2} (\mathbb{R}^n)$ space.  To go around this difficulty we shall first take a limit in a very weak sense, and then we will obtain a sharp uniform bound in $H^1 (\mathbb{R}^n)$ space by some tricky analysis. It will enable us to improve the convergence result to $H^1 (\mathbb{R}^n)$ space. 

The rest of the paper is organized as follows. In Section 2 we consider the extension problem which localizes the nonlocal problem \eqref{eq-i1}. In Section 3 we find the ground state solutions. Section 4 is devoted to prove the positiveness and the symmetry result of the ground state solutions. In Section 5 we shall obtain the important uniform boundedness of the ground states solutions. Then we shall prove the nonrelativistic limit in Section 6.


\section{The extension problem}
We review the well known fact (see \cite{CN2}) that for any $u \in H^{1/2}(\mathbb{R}^n)$, there exists a unique solution $U \in H^1(\mathbb{R}^{n+1}_+)$ of the following boundary value problem
\begin{equation}\label{extension-dirichlet}
\begin{array}{rcl}
(-c^2\Delta_{x,y} +m^2c^4)U(x,y) & = & 0  \quad \text{ in } \mathbb{R}^{n+1}_+ := \mathbb{R}^n \times (0,\infty), \\
U(x,0) & = & u(x) \quad \text{ in } \partial\mathbb{R}^{n+1}_+ = \mathbb{R}^n
\end{array}
\end{equation}
such that
\begin{equation}\label{normal-der}
-c\frac{\partial U}{\partial y}(x,0) = \sqrt{-c^2\Delta+m^2c^4}\,u (x),\quad x \in \mathbb{R}^n
\end{equation}
in distribution sense.
Then for any solution $u \in H^{1/2}(\mathbb{R}^n)$ of \eqref{main-eq}, a unique solution $U$ of \eqref{extension-dirichlet}
with boundary value $u$ satisfies
\begin{equation}\label{extension-neumann}
\begin{array}{rcl}
(-c^2\Delta_{x,y} +m^2c^4)U(x,y) & = & 0  \quad \text{ in } \mathbb{R}^{n+1}_+ \\
-c\frac{\partial U}{\partial y}(x,0) & = & (mc^2-\mu)u(x) + |u|^{p-2}u(x) \quad \text{ in } \partial\mathbb{R}^{n+1}_+ = \mathbb{R}^n
\end{array}
\end{equation}
and thus is a critical point of the functional
\begin{multline}
I_e(U) := \frac{1}{2c}\int_{\mathbb{R}^{n+1}_+}c^2|\nabla U(x,y)|^2+m^2c^4U(x,y)^2\,dxdy \\
+\frac{(-mc^2+\mu)}{2}\int_{\mathbb{R}^n}U(x,0)^2\,dx -\frac{1}{p}\int_{\mathbb{R}^n}|U(x,0)|^p\,dx.
\end{multline}
Conversely, if $U \in H^1(\mathbb{R}^{n+1}_+)$ is a critical point of $I_e$, then its trace $U(x,0)$ belongs to
$H^{1/2}(\mathbb{R}^n)$ and is a critical point of $I$ defined in \eqref{eq-E}.

\begin{lem}\label{lem-trace-ineq}
Let $V \in H^1(\mathbb{R}^{n+1}_+)$ and $v(x) := V(x,0)$. Then, it holds that
\begin{equation}\label{trace-ineq}
\int_{\mathbb{R}^n}(c^2|\xi|^2+m^2c^4)^{1/2}|\hat v|^2\,d\xi \leq \frac{1}{c}\int_{\mathbb{R}^{n+1}_+}c^2|\nabla V|^2 +m^2c^4V^2\,dxdy.
\end{equation}
The equality holds if and only if $V$ satisfies \eqref{extension-dirichlet} with boundary value $v$.
\end{lem}
\begin{proof}
If $V$ satisfies \eqref{extension-dirichlet} with boundary value $v$, then we deduce by testing $V$ to
\eqref{extension-dirichlet} and using \eqref{normal-der} that the equality of \eqref{trace-ineq} holds.
Suppose that $V$ is an arbitrary function in $H^1(\mathbb{R}^{n+1}_+)$.
Let $W_0$ be a unique minimizer of the problem
\[
\min\left\{ \int_{\mathbb{R}^{n+1}_+}c^2|\nabla W|^2 +m^2c^4W^2\,dxdy ~|~ W(x,0) = v(x) \right\},
\]
which is a solution of \eqref{extension-dirichlet} with boundary data $v(x)$.
Then,
\[
\int_{\mathbb{R}^n}(c^2|\xi|^2+m^2c^4)^{1/2}|\hat v|^2\,d\xi
=\frac1c\int_{\mathbb{R}^{n+1}_+}c^2|\nabla W_0|^2 +m^2c^4W_0^2\,dxdy
\leq \frac1c\int_{\mathbb{R}^{n+1}_+}c^2|\nabla V|^2 +m^2c^4V^2\,dxdy.
\]
This completes the proof.
\end{proof}

The following fractional Sobolev embedding is well-known in literature.
\begin{lem}\label{fractional-embedding}
For any $q \in [2, 2n/(n-1)]$, $H^{1/2}(\mathbb{R}^n)$ is continuously embedded in $L^q(\mathbb{R}^n)$.
In other words, one has
\[
\left(\int_{\mathbb{R}^n}|v|^q\,dx\right)^{1/q}
\leq C\left(\int_{\mathbb{R}^n}(|\xi|^2+1)^\frac12|\hat{v}|^2\,d\xi\right)^{1/2}.
\]
Let $H^{1/2}_r(\mathbb{R}^n)$ be the set of radially symmetric functions in $H^{1/2}(\mathbb{R}^n)$.
Then for any $q \in (2, 2n/(n-1))$, the embedding $H^{1/2}_r(\mathbb{R}^n) \hookrightarrow L^q(\mathbb{R}^n)$
is compact.
\end{lem}

We say $U$ is a ground state solution of \eqref{extension-neumann} if $U$ is a critical point of $I_e$ and satisfies
\[
I_e(U) = \min\{ I_e(V) ~|~ V \in H^1(\mathbb{R}^{n+1}_+) \setminus \{0\},\, I_e'(V) = 0 \}.
\]

\begin{prop}\label{equivalence}
Let $U$ be a ground state solution of \eqref{extension-neumann}. Then $U(x,0)$ is a ground state solution
of $\eqref{main-eq}$.
Conversely, let $u$ be a ground state solution of \eqref{main-eq}. Then, a unique solution $U$ of \eqref{extension-dirichlet} with boundary data $u$ is a ground state solution of \eqref{extension-neumann}.
\end{prop}
\begin{proof}
Let $S$ and $S_e$ be the sets of critical points of $I$ and $I_e$ respectively.
Then the trace map $T: S_e \subset H^1(\mathbb{R}^{n+1}_+) \to S \subset H^{1/2}(\mathbb{R}^n)$
given by $T(U) = U(x,0)$ is an isometry.
To see this, note first that for any $u \in S$, the inverse image of $T$ exists and is given by
the unique solution of $U$ of \eqref{extension-dirichlet} with boundary value $u$.
Then Lemma \ref{lem-trace-ineq} says that $T$ is an isometry between $S_e$ and $S$.
Therefore, we have
\[
I(T(U)) = I_e(U) \quad \text{ or equivalently } \quad I(u) = I_e(T^{-1}(u)),
\]
which proves the proposition.
\end{proof}

\section{Existence result}
Fix $c \geq 1,\, m,\,\mu > 0$. For $U \in H^{1/2}(\mathbb{R}^{n+1}_+)$, we define a norm
\begin{equation}
\|U\|^2 = \int_{\mathbb{R}^{n+1}_+} c^2|\nabla U(x,y)|^2+m^2c^4U(x,y)^2\, dxdy
-c(mc^2-\mu)\int_{\mathbb{R}^n}U(x,0)^2\,dx
\end{equation}
which is equivalent to the standard norm
\[
\left( \int_{\mathbb{R}^{n+1}} |\nabla U(x,y)|^2+U(x,y)^2\, dxdy \right)^{1/2}
 \]
of the Sobolev space $H^1(\mathbb{R}^{n+1}_+)$. It is definite when $\mu \geq mc^2$. So, we assume that $\mu < mc^2$.
To show the equivalence we first recall from \cite{CN2} the inequality
\[
\int_{A}U(x,0)^2\,dx \leq 2\left(\int_{A\times[0,\infty)}U(x,y)^2\,dxdy\right)^{1/2}
\left(\int_{A\times[0,\infty)}|\nabla U(x,y)|^2\,dxdy\right)^{1/2}
\]
where $A \subset \mathbb{R}^n$ is any measurable set.
Applying Young's inequality,
we have 
\begin{equation}\label{ineq-1}
\int_{A\times[0,\infty)} c^2 |\nabla U (x,y)|^2 + m^2 c^4 U(x,y)^2\, dx dy \geq m c^3 \int_{A} U(x,0)^2 dx.
\end{equation}
Let $\alpha= \frac{c(mc^2 - \mu )}{mc^3} \in (0,1)$. Then, using the above inequality with $A = \mathbb{R}^n$,
we have
\begin{equation}
\begin{split}
&\|U\|^2 =\int_{\mathbb{R}^{n+1}_{+}}c^2 |\nabla U (x,y)|^2 + m^2 c^4 U (x,y)^2\, dx dy - c (mc^2 - \mu) \int_{\mathbb{R}^n} U(x,0)^2\, dx
\\
&\qquad\quad \geq (1-\alpha) \int_{\mathbb{R}^{n+1}_{+}}c^2 |\nabla U (x,y)|^2 + m^2 c^4 U (x,y)^2\, dx dy,
\end{split}
\end{equation}
which shows the equivalence.

Now, we prove Theorem \ref{existence}.
By Proposition \ref{equivalence}, we may find a ground state solution of \eqref{extension-neumann}.
We shall achieve this by finding a minimizer of $I_e$ on the Nehari manifold
\begin{equation}
\mathcal{N}_e := \{V \in H^1(\mathbb{R}^{n+1}_+) ~|~ J_e(V) = 0,\, V \not\equiv 0 \},
\end{equation}
where $J_e(V) := I'_e(V)V$, i.e.,
\begin{equation}\label{eq-je}
\begin{split}
J_e(V) =& \frac{1}{c}\int_{\mathbb{R}^{n+1}_{+}}c^2 |\nabla V (x,y)|^2 + m^2 c^4 V(x,y)^2 dx dy
\\
&\qquad +  (-mc^2 + \mu) \int_{\mathbb{R}^n} V(x,0)^2 dx -  \int_{\mathbb{R}^n} |V(x,0)|^p dx.
\end{split}
\end{equation}

\begin{lem}\label{nonmepty}
$\mathcal{N}_e$ is nonempty.
\end{lem}
\begin{proof}
Take any nonzero $V \in H^1(\mathbb{R}^{n+1}_+)$. Then as a function of $t$, it is easy to see that
\[
I_e(tV) := \frac{t^2}{2c}\left(\int_{\mathbb{R}^{n+1}_+}c^2|\nabla V|^2+m^2c^4V^2\,dxdy
-c(mc^2-\mu)\int_{\mathbb{R}^n}V(x,0)^2\,dx\right)
-\frac{t^p}{p}\int_{\mathbb{R}^n}|V(x,0)|^p\,dx.
\]
attains its strict local maximum at some $t_0 \in (0,\,\infty)$.
By differentiating with respect to $t$, we get $I_e'(t_0 V)V = 0$, which says $t_0 V \in \mathcal{N}_e$.
\end{proof}

\begin{prop}\label{minimizer}
Let
\[
M_e : = \min_{V \in \mathcal{N}_e}I_e(V).
\]
Then there exists radially symmetric $U \in \mathcal{N}_e$ such that $I_e(U) = M_e$.
\end{prop}
\begin{proof}
We claim that there exists a minimizing sequence $\{V_j\} \in \mathcal{N}_e$ of $I_e$
such that for any fixed $y > 0$, every $V_j$ is radially symmetric with respect to $x \in \mathbb{R}^n$.
To see this, take first an arbitrary minimizing sequence $\{V_j\}$.
Define $v_j(x) = V_j(x,0)$.
Let $\tilde{v}_j$ be the symmetric rearrangement of $v_j$ (see \cite{LL} for the definition).
It is shown in \cite{LL} that
\[
\int_{\mathbb{R}^n}(\sqrt{c^2 \xi^2 + m^2 c^4})|\hat{\tilde{v}}_j(\xi)|^2\,d\xi \leq
\int_{\mathbb{R}^n}(\sqrt{c^2 \xi^2 + m^2 c^4})|\hat{{v}}_j(\xi)|^2\,d\xi
\]
and for all $q \geq 1$,
\[
\int_{\mathbb{R}^n} |\tilde{v}_j|^q\,dx = \int_{\mathbb{R}^n} |v_j|^q\,dx.
\]
Let $\tilde{V}_j$ be the unique solution of \eqref{extension-dirichlet} with boundary data $\tilde{v}_j$.
Then one has from Lemma \ref{lem-trace-ineq}
\[
\begin{aligned}
\int_{\mathbb{R}^{n+1}_+}c^2|\nabla \tilde{V}_j|^2+m^2c^4\tilde{V}_j^2\,dxdy
& = \int_{\mathbb{R}^n}(\sqrt{c^2 \xi^2 + m^2 c^4})|\hat{\tilde{v}}_j(\xi)|^2\,d\xi \\
&\leq \int_{\mathbb{R}^n}(\sqrt{c^2 \xi^2 + m^2 c^4})|\hat{{v}}_j(\xi)|^2\,d\xi
\leq \int_{\mathbb{R}^{n+1}_+}c^2|\nabla V_j|^2+m^2c^4V_j^2\,dxdy.
\end{aligned}
\]
Thus, we have $J_e(\tilde{V}_j) \leq J_e(V_j) = 0$ and $I_e(\tilde{V}_j) \leq I_e(V_j)$.
As in the proof of Lemma \ref{nonmepty}, we can find $t_j > 0$ such that $J_e(t_j\tilde{V}_j) = 0$, and it can be checked easily that $t_j \in (0, 1]$ because $J_e(\tilde{V}_j) \leq 0$ means that
\[
\int_{\mathbb{R}^{n+1}_+}c^2|\nabla \tilde{V}_j|^2+m^2c^4\tilde{V}_j^2\,dxdy-c(mc^2-\mu)\int_{\mathbb{R}^n}\tilde{v}_j^2\,dx
\leq c\int_{\mathbb{R}^n}|\tilde{v}_j|^p\,dx.
\]
From the fact that $J_{e} (t_j \tilde{V}_j) =0$ and $t_j \in (0,1]$, we deduce that
\[
I_e(t_j\tilde{V}_j) = \left(\frac{1}{2}-\frac{1}{p}\right)\int_{\mathbb{R}^n}|t_j\tilde{v}_j|^p\,dx
\leq \left(\frac{1}{2}-\frac{1}{p}\right)\int_{\mathbb{R}^n}|\tilde{v}_j|^p\,dx
= \left(\frac{1}{2}-\frac{1}{p}\right)\int_{\mathbb{R}^n}|v_j|^p\,dx = I_e(V_j).
\]
Therefore $\{t_j\tilde{V}_j\}$ is a desired minimizing sequence of $I_e$ on $\mathcal{N}_e$.

Now, we just denote $\{t_j\tilde{V}_j\}$ by $\{V_j\}$.
Since
\[
\frac{1}{c}\left(\frac{1}{2}-\frac{1}{p}\right)
\left(\int_{\mathbb{R}^{n+1}_+}c^2|\nabla V_j|^2+m^2c^4V_j^2\,dxdy-(mc^2-\mu)\int_{\mathbb{R}^n}v_j^2\,dx\right)
 = I_e(V_j) < C,
\]
$\|V_j\|$ is uniformly bounded for $j$.
Then there exists some $V_0 \in H^1(\mathbb{R}^{n+1}_+)$
such that $V_j \rightharpoonup V_0$ weakly in $H^1(\mathbb{R}^{n+1}_+)$
and $V_j(x,0) \to V_0(x,0)$ strongly in $L^q(\mathbb{R}^n)$ for all $q \in (2, \frac{2n}{n-1})$ due to the Lemma \ref{fractional-embedding}.
From the weakly lower semi-continuity of the functional $I_e$ and $J_e$, we deduce
$I_e(V_0) \leq M_e$ and $J_e(V_0) \leq 0$.
We claim that $V_0 \not\equiv 0$. If not, $V_j(x,0) \to 0$ in $L^p(\mathbb{R}^n)$ but since $J_e(V_j) = 0$, by definition \eqref{eq-je} of $J_{e}$ we have
\begin{equation}
\begin{split}
 c \int_{\mathbb{R}^n} |V_j (x,0)|^p dx &=\int_{\mathbb{R}^{n+1}_{+}}c^2 |\nabla V_j  (x,y)|^2 + m^2 c^4 V_j  (x,y)^2 dx dy
\\
&\qquad + c (-mc^2 + \mu) \int_{\mathbb{R}^n} V_j (x,0)^2 dx
\\
& \geq  (1-\alpha)  \int_{\mathbb{R}^{n+1}_{+}}c^2 |\nabla V_j  (x,y)|^2 + m^2 c^4 V_j  (x,y)^2 dx dy
\\
& \geq C \| V_j (\cdot, 0)\|_{L^p (\mathbb{R}^n)}^2,
\end{split}
\end{equation}
where we applied the fractional Sobolev embedding in the second inequality. This shows that $\|V_j(\cdot,0)\|_{L^p(\mathbb{R}^n)}$ is bounded away from $0$. Therefore $V_0 \not\equiv 0$.
Then as above, there is $t_0 \in (0, 1]$ such that
$I_e(t_0V_0) \leq M_e$ and $J_e(t_0V_0) = 0$.
This proves the proposition.
\end{proof}

Now, we are ready to complete the proof of Theorem \ref{existence}.
Let $U$ be the minimizer obtained in Proposition \ref{minimizer}.
Since $U \in \mathcal{N}_e$,
\[
\begin{aligned}
J_e'(U)U &=  \frac{2}{c}\left(\int_{\mathbb{R}^{n+1}_+}c^2|\nabla U|^2+m^2c^4U^2\,dxdy
-c(mc^2-\mu)\int_{\mathbb{R}^n}U (x,0)^2\,dx\right) - p\int_{\mathbb{R}^n} |U(x,0)|^p\,dx \\
&= (2-p)\int_{\mathbb{R}^n} |U(x,0)|^p\,dx \neq 0
\end{aligned}
\]
so $J_e'(U) \neq 0$. Then the Lagrange multiplier rule applies to see that for some $\lambda \in \R$,
\begin{equation}\label{LM}
I_e'(U) = \lambda J_e'(U).
\end{equation}
By testing $U$ to \eqref{LM}, we immediately see $\lambda = 0$ so that $U$ is a nontrivial solution of \eqref{main-eq}.
Moreover, $U$ is indeed a ground state because every critical point $V$ of $I_e$ belongs to $\mathcal{N}_e$.
This proves Theorem \ref{existence}.

\section{Sign definiteness and symmetry of ground states}
This section is devoted to prove Theorem \ref{qp}.
We first prove ground state solutions have one sign.
\begin{prop}
Every ground state solution $u$ of \eqref{main-eq} has one sign.
In other words, it is either positive everywhere or negative everywhere.
\end{prop}
\begin{proof}
Let $U$ be a unique solution of \eqref{extension-dirichlet} with boundary data $u$.
Then, by Proposition \ref{equivalence}, $U \in \mathcal{N}_e$ and $I_e(U) = M_e$.
Note that $|\nabla U|^2 = |\nabla|U||^2$. To see this, define $U_+ := \max\{U,\, 0\}$
and $U_- := \max\{-U,\,0\}$ so that $U = U_+-U_-$ and $|U| = U_+ + U_-$
and use the fact that $\nabla U_+\cdot\nabla U_- = 0$.
This shows $|U| \in \mathcal{N}_e$ and $I_e(|U|) = M_e$, which means that $|U|$ is also a ground state
solution of \eqref{extension-neumann}.
Applying the standard elliptic regularity theory (see \cite{GT} and \cite{CN2}),
we see $|U|$ is a classical solution of \eqref{extension-neumann} and
$|U| \in C^{2,\alpha}(\overline{\mathbb{R}^{n+1}_+})$.

Now, arguing indirectly, suppose that $u$ changes sign. Then there exists $x_0 \in \mathbb{R}^n$ such that
$u(x_0) = 0$ and consequently, $|U|$ attains it global minimum on $\overline{\mathbb{R}^{n+1}_+}$ at $(x_0,0)$.
By using the Hopf maximum principle, we see $(\partial|U|/\partial\nu)(x_0,0) < 0$.
But this makes a contradiction because \eqref{extension-neumann} says
\[
-c\frac{\partial}{\partial y}|U|(x_0,0) = (mc^2-\mu)|u|(x_0) + |u|^{p-1}(x_0) = 0.
\]
This completes the proof.
\end{proof}

Without loss of generality, we may assume that every ground state solution of \eqref{main-eq}
is positive everywhere since $u$ is a solution of \eqref{main-eq} if and only if $-u$ is a solution of of \eqref{main-eq}.
We complete the proof of Theorem \ref{qp} by proving the following proposition.
\begin{prop}\label{symmetry}
Every positive solution $u \in H^{1/2}(\mathbb{R}^n)$ of \eqref{main-eq} is radially symmetric on $\mathbb{R}^n$
with respect to some point $x_0 \in \mathbb{R}^n$.
\end{prop}
%
\begin{proof}
Let $U$ be a unique solution of \eqref{extension-dirichlet} with boundary data $u$,
which is consequently a classical solution of \eqref{extension-neumann} and belongs to $C^2(\overline{\mathbb{R}^{n+1}_+})$. The strong maximum principle says that $U$ is positive on $\overline{\mathbb{R}^{n+1}_+}$.
By the standard elliptic regularity theory (see \cite{GT} and \cite{CN2}), we know
\[
\lim_{|x| + y \to \infty}|U(x,y)| = 0.
\]
Take $\lambda > 0$. Define
\[
M_\lambda :=\{ (\lambda, x_2, \dots, x_n, y) \in \mathbb{R}^{n+1}_+ ~|~ x_i \in \R\,(i=2,\dots, n),\, y \geq 0\}
\]
and
\[
R_\lambda := \{(x_1, \dots, x_n, y) \in \mathbb{R}^{n+1}_+ ~|~ x_1 > \lambda,\, x_i \in \R\, (i=2,\dots, n),\, y \geq 0 \}.
\]
Let $U_\lambda(x_1,\dots, x_n, y) := U(2\lambda -x_1, x_2, \dots, x_n, y)$ be a reflection of $U$ with respect to  $M_\lambda$.
Then $W_\lambda := U_\lambda - U$ satisfies
\begin{eqnarray}\label{eq-compare}
(-c^2\Delta_{x,y} +m^2c^4)W_\lambda(x,y) &=& 0 \qquad\textrm{in } \mathbb{R}^{n+1}_+  \\ \nonumber
-c\frac{\partial W_\lambda}{\partial y}(x,0) & = & (C_\lambda(x)+mc^2-\mu)w_\lambda(x)  \qquad\textrm{in } \mathbb{R}^n,
\\ \nonumber
\end{eqnarray}
where $u_\lambda(x) := U_\lambda(x,0),\, w_\lambda(x) := W_\lambda(x,0)$ and
\[
C_\lambda(x) = \frac{u_\lambda^{p-1}(x) - u^{p-1}(x)}{u_\lambda(x) - u(x)}.
\]
Let $W_\lambda^- := \min\{0,W_\lambda\}$ be the negative part of $W_\lambda$. Note that as $\lambda \to \infty$, $C_\lambda(x) \to 0$ uniformly for $x$ with $W_{\lambda}^{-}(x)\neq 0$  due to the fact
$\lim_{|x| + y \to \infty}|U(x,y)| = 0$ and $0 \leq U_{\lambda} < U$ whenever $W_{\lambda}^{-} \neq 0$. Then, testing $W_\lambda^-$ to \eqref{eq-compare} and applying inequality \eqref{ineq-1}, we get
for sufficiently large $\lambda > 0$,
\begin{equation}
\begin{aligned}
&\int_{R_\lambda}c^2|\nabla W_\lambda^-|^2 +m^2c^4(W_\lambda^-)^2\,dxdy
 = \int_{\{x_1 > \lambda\}}c(C_\lambda(x)+mc^2-\mu)(w_\lambda^-)^2\,dx \\
& \leq \frac{c}{mc}(mc^2-\frac\mu2)\int_{R_\lambda}|\nabla W_\lambda^-|^2\,dxdy
+mc^2(mc^2-\frac\mu2)\int_{R_\lambda}(W_\lambda^-)^2\,dxdy,
\end{aligned}
\end{equation}
which tells us that $W_\lambda^- \equiv 0$ on $R_\lambda$.
Thus we conclude that if $\lambda > 0$ is sufficiently large,
\[
W_\lambda(x,y) \geq 0 \quad \text{ on } R_\lambda.
\]
Now, define
\[
\nu := \inf\{s > 0 ~|~ W_\lambda \geq 0 \text{ on } R_\lambda \text{ for every } \lambda \geq s\}.
\]
\\
The case $\nu > 0$ : We claim that $W_\nu \equiv 0$. To the contrary, suppose not.
Then we can see that $W_\nu > 0$ on the set
\[
R'_\nu : = \{(x_1, \dots, x_n, y) ~|~ x_1 > \nu,\, x_i \in \R(i = 2,\dots,n),\, y > 0 \}
\]
from the strong maximum principle the fact that $W_\nu \geq 0$ on $R_\nu$ holds by continuity. We also have $w_{\nu}(x) \geq 0$ on the set $\{x \in \mathbb{R}^n ~|~x_1 \geq \nu\}$ by continuity. Furthermore, we have $w_\nu(x) > 0$ on the set $\{ x \in \mathbb{R}^n ~|~ x_1 > \nu\}$.
Otherwise there exists $x_0 \in \mathbb{R}^n$ such that $(x_0)_1 > \nu$ and $w_\nu(x_0) = 0$.
Then the Hopf maximum principle says $-\frac{\partial}{\partial y}W_\nu(x_0,0) < 0$
but this contradict with \eqref{eq-compare} since $w_\nu(x_0) = 0$.
Take $\lambda_j < \nu$ such that $\lambda_j \to \nu$ as $j \to \infty$.
Let $r_0$ be a positive number such that $|C_\nu(x)| \leq \mu/4$ for every $|x| > r_0$.
Since $\|U_{\lambda_j}\|_{C^1(\mathbb{R}^{n+1}_+)}$ is uniformly bounded,
$D := \|C_{\lambda_j}\|_{L^\infty(\mathbb{R}^n)} < \infty$
and $|C_{\lambda_j}(x)| \leq \mu/2$ for every $|x| > r_0$ and $j \in \mathbb{N}$.
Let $B_{r_0}(p_j)$ be the $n$ dimensional Euclidean ball with center $p_j$ and radius $r_0$
where $p_j = (\lambda_j,0\dots,0)$.
As above, we also have
\begin{equation}\label{eq-compare2}
\begin{aligned}
\int_{R_{\lambda_j}}c^2|\nabla W_{\lambda_j}^-|^2 +m^2c^4(W_{\lambda_j}^-)^2\,dxdy
&\leq \int_{\{x_1 > {\lambda_j}\}}c(C_{\lambda_j}(x)+mc^2-\mu)(w_{\lambda_j}^-)^2\,dx \\
& \leq c(D+mc^2+\mu)\int_{\{x_1 > {\lambda_j}\} \cap B_{r_0}(p_j)}(w_{\lambda_j}^-)^2\,dx \\
&\qquad+c(mc^2-\frac12\mu)\int_{\{x_1 > {\lambda_j}\} \setminus B_{r_0}(p_j)}(w_{\lambda_j}^-)^2\,dx.
\end{aligned}
\end{equation}
As before, the second term in right-hand side of \eqref{eq-compare2} is absorbed in left-hand side.
Since $w_\nu(x) > 0$ on the set $\{ x \in \mathbb{R}^n ~|~ x_1 > \nu\}$, the measure of set $E_j$,
support of $w_{\lambda_j}^-$ on $B_{r_0}(p_j)$, goes to $0$ as $j \to \infty$.
Then using H\"older and Sobolev inequality, we see
\[
\begin{aligned}
\int_{\{x_1 > {\lambda_j}\} \cap B_{r_0}(p_j)}(w_{\lambda_j}^-)^2\,dx &=
\int_{\{x_1 > {\lambda_j}\}}\chi_{E_j}(w_{\lambda_j}^-)^2\,dx
\leq \|\chi_{E_j}\|_{L^n}\|w_{\lambda_j}^-\|_{L^{\frac{2n}{n-1}}}^2 \\
&\leq o(1)\int_{R_{\lambda_j}}|\nabla W_{\lambda_j}^-|^2\,dxdy.
\end{aligned}
\]
Therefore if $j$ is large, we see $W_{\lambda_j} \geq 0$ on $R_{\lambda_j}$ \
but this contradicts with the minimality of $\nu$.
Thus we finally conclude that $W_\nu \equiv 0$ on $R_\nu$ and get the symmetry in the $x_1$ direction. \\

\noindent The case $\nu = 0$ : We repeat the above argument for $\lambda < 0$ and $W_\lambda := U_\lambda -U$ defined on
\[
L_\lambda := \{(x_1, \dots, x_n, y) \in \mathbb{R}^{n+1}_+ ~|~ x_1 < \lambda,\, x_i \in \R(i=2,\dots, n),\, y \geq 0 \}.
\]
Then we see that $W_\lambda \geq 0$ for sufficiently large $|\lambda|$.
Define
\[
\nu' := \sup\{s < 0 ~|~ W_\lambda \geq 0 \text{ on } L_\lambda \text{ for every } \lambda \leq s\}.
\]
If $\nu' < 0$, we get the symmetry as above. If $\nu' = 0$, we get from $\nu = 0$ that
\[
U(-x_1, x_2, \dots, x_n, y) \geq U(x_1, \dots, x_n, y) \quad \text{ on } R^{n+1}_+.
\]
Consequently, by replacing $x_1$ with $-x_1$ we see the symmetry,
\[
U(-x_1, x_2, \dots, x_n, y) = U(x_1, \dots, x_n, y) \quad \text{ on } R^{n+1}_+.
\]
To complete the entire proof, we repeat the above process for every other directions $x_i$. This complete the proof.
\end{proof}

\section{Uniform estimates of the ground state solutions in $H^{1}(\mathbb{R}^n)$}
In what follows, we shall denote by $u_c$ a radially symmetric positive ground state solution to \eqref{main-eq} for each $c \geq 1$.
The aim of this section is to show that $\{u_c\}_{c >1}$ found in Section 3 are uniformly upper and lower bounded in $L^{p}(\mathbb{R}^n)$
and eventually upper bounded in $H^1(\mathbb{R}^n)$.
This will be essentially used when we obtain the nonrelativstic limit in the next section.
To clearly show the dependence of $c$,
we add a subscript $c$ to the functionals $I_e$, $J_e$ and the space $\mathcal{N}_e$
introduced in Section 2 and Section 3 so that
\begin{equation}
\begin{split}
I_{e, c}(U) &= \frac{1}{2c}\int_{\mathbb{R}^{n+1}_+}c^2|\nabla U(x,y)|^2+m^2c^4U(x,y)^2\,dxdy
\\
&\quad \qquad +\frac{(-mc^2+\mu)}{2}\int_{\mathbb{R}^n}U(x,0)^2\,dx -\frac{1}{p}\int_{\mathbb{R}^n}|U(x,0)|^p\,dx,
\\
J_{e,c}(U) &=\frac{1}{c}\int_{\mathbb{R}^{n+1}_{+}}c^2 |\nabla U (x,y)|^2 + m^2 c^4 U (x,y)^2 dx dy
\\
&\qquad\quad +  (-mc^2 + \mu) \int_{\mathbb{R}^n} U(x,0)^2 dx -  \int_{\mathbb{R}^n} |U(x,0)|^p dx,
\\
\mathcal{N}_{e,c} &= \{V \in H^1(\mathbb{R}^{n+1}_+) ~|~ J_{e,c}(V) = 0,\, V \not\equiv 0 \}.
\end{split}
\end{equation}
Likewise,
\[I_c := I = \frac{1}{2} \int_{\mathbb{R}^n}(\sqrt{c^2 \xi^2 + m^2 c^4}-mc^2+\mu)|\hat{u} (\xi)|^2\,d\xi
-\frac{1}{p}\int_{\mathbb{R}^n} |u|^p\,dx.
\]

\begin{lem}\label{lem-iuc}
We have  $\sup_{c \geq 1} I_c(u_c) < \infty$.
\end{lem}
\begin{proof}
Recall that we obtained the solution $u_c$ in Section 3 by finding a minimizer $U_c$ of $I_{e,c}$ on the Nehari manifold $\mathcal{N}_{e,c}$. In fact, $u_c(x) = U_c(x,0)$ and $I_c(u_c) = I_{e,c}(U_c)$.
Note that any function $V \in \mathcal{N}_{e,c}$ satisfies
\begin{equation}
 \int_{\mathbb{R}^n} |V(x,0)|^p dx =\frac{1}{c} \int_{\mathbb{R}^{n+1}_{+}} c^2 |\nabla V(x,y)|^2 + m^2 c^4 V(x,y)^2 dx dy +  (-mc^2 +\mu) \int_{\mathbb{R}^n}V(x,0)^2 dx.
\end{equation}
Using this property, we have
\begin{equation}\label{eq-ju-1}
\begin{split}
&\left(\frac{1}{2} -\frac{1}{p}\right)^{-1} I_{e,c} (U_{c})
=\min_{V \in \mathcal{N}_{e,c}} \left(\frac{1}{2} -\frac{1}{p}\right)^{-1}  I_{e,c} (V)
= \min_{V \in \mathcal{N}_{e,c}}  \int_{\mathbb{R}^n} |V(x,0)|^p dx
\\
&= \min_{V \in \mathcal{N}_{e,c}} \frac{ \left[ \frac{1}{c} \int_{\mathbb{R}^{n+1}_{+}} c^2 |\nabla V(x,y)|^2 + m^2 c^4 V(x,y)^2 dx dy +  (-mc^2 +\mu) \int_{\mathbb{R}^n}V(x,0)^2 dx\right]^{\frac{p}{p-2}}}{ \left[  \int_{\mathbb{R}^n} |V(x,0)|^p dx\right]^{\frac{2}{p-2}}}
\\
&= \min_{V \in \mathcal{N}_{e,c}}\frac{ \left[ \frac{1}{c} \int_{\mathbb{R}^{n+1}_{+}} c^2 |\nabla (tV)(x,y)|^2 + m^2 c^4 (tV)(x,y)^2 dx dy +  (-mc^2 +\mu) \int_{\mathbb{R}^n}(tV)(x,0)^2 dx\right]^{\frac{p}{p-2}}}{ \left[  \int_{\mathbb{R}^n} |(tV)(x,0)|^p dx\right]^{\frac{2}{p-2}}}
\\
&= \min_{V \in H^1 (\mathbb{R}^{n+1}_{+})}\frac{ \left[  \frac{1}{c}\int_{\mathbb{R}^{n+1}_{+}} c^2 |\nabla V(x,y)|^2 + m^2 c^4 V(x,y)^2 dx dy +  (-mc^2 +\mu) \int_{\mathbb{R}^n}V(x,0)^2 dx\right]^{\frac{p}{p-2}}}{ \left[ \int_{\mathbb{R}^n} |V(x,0)|^p dx\right]^{\frac{2}{p-2}}}
\\
&=\min_{v \in H^{1/2} (\mathbb{R}^n)}\frac{ \left[  \int_{\mathbb{R}^n}(c^2 |\xi|^2 + m^2 c^4)^{\frac{1}{2}} |\hat{v}|^2 d \xi + (-mc^2 +\mu) \int_{\mathbb{R}^n}v(x)^2 dx\right]^{\frac{p}{p-2}}}{ \left[ \int_{\mathbb{R}^n} |v(x)|^p dx\right]^{\frac{2}{p-2}}},
\end{split}
\end{equation}
where we used Lemma \ref{lem-trace-ineq} in the last equality. Take any nonzero function $\phi \in \mathcal{S}(\mathbb{R}^n)$. Then we see that
\begin{equation}
\begin{split}
\int_{\mathbb{R}^n} \left(\sqrt{c^2 |\xi|^2 + m^2 c^4}- mc^2 \right) |\hat{\phi}(\xi)|^2 d\xi &= \int_{\mathbb{R}^n} \frac{|\xi|^2}{\sqrt{ \frac{|\xi|^2}{c^2}+m^2} + m} |\hat{\phi}(\xi)|^2 d\xi
\\
& \leq \frac{1}{2m}\int_{\mathbb{R}^n}  |\xi|^2 |\hat{\phi}(\xi)|^2 d\xi.
\end{split}
\end{equation}
Therefore,
\begin{equation}
\sup_{c > 1}~\frac{ \left[  \int_{\mathbb{R}^n}(c^2 |\xi|^2 + m^2 c^4)^{\frac{1}{2}} |\hat{\phi}|^2 d \xi + (-mc^2 +\mu) \int_{\mathbb{R}^n}\phi(x)^2 dx\right]^{\frac{p}{p-2}}}{ \left[ \int_{\mathbb{R}^n} |\phi(x)|^p dx\right]^{\frac{2}{p-2}}} < \infty.
\end{equation}
Combining this with \eqref{eq-ju-1} we conclude that $\sup_{c > 1} I_{e, c}(u_c) < \infty$. The lemma is proved.
\end{proof}
\begin{lem}\label{lem-upc}
For the ground state solutions $\{u_c\}_{c >1}$ to \eqref{main-eq}
 there exists a constant $C = C(m, \mu, n)$ such that
\begin{equation}
\frac{1}{C} \leq \int_{\mathbb{R}^n} u_c^{p} (x) dx \leq C.
\end{equation}
In addition, we also have
$\sup_{c>1}  \|u_c\|_{H^{1/2}(\mathbb{R}^n)} < \infty.$
\end{lem}
\begin{proof}
Using Lemma \ref{lem-iuc} we have
\begin{equation}
\begin{split}
\sup_{c>1} \int_{\mathbb{R}^n} |u_c|^p dx& =\sup_{c>1} \int_{\mathbb{R}^n} c \left( \sqrt{ |\xi|^2 + m^2 c^2}- mc\right) |\hat{u}_c (\xi)|^2 + \mu |\hat{u}_{c}(\xi)|^2 d \xi
\\
&=\sup_{c>1}\left( \frac{1}{2}-\frac{1}{p}\right)^{-1} I_{e,c} (u_c)< \infty.
\end{split}
\end{equation}
Hence we have the uniform upper bound of $\int_{\mathbb{R}^n} |u_c|^{p} dx$ for $c>1$. Next, we use some trivial estimates and the Sobolev embedding $H^{1/2} (\mathbb{R}^n) \rightarrow L^{p}(\mathbb{R}^n)$ for $p < \frac{2n}{n-1}$ to get 
\begin{equation}\label{eq-up-1}
\begin{split}
\int_{\mathbb{R}^n} u_c^{p}(x) dx & = \int_{\mathbb{R}^n} c \left( \sqrt{\xi^2 +m^2 c^2}-m c\right) |\hat{u}_c (\xi)|^2 + \mu |\hat{u}_c (\xi)|^2 d\xi
\\
 &= \int_{\mathbb{R}^n} \left( \frac{|{\xi}|^2}{\sqrt{|\frac{\xi}{c}|^2 +m} +m} +\mu\right) |\hat{u}_c (\xi)|^2 d\xi
\\
& \geq \int_{\mathbb{R}^n} \left( \frac{|\xi|^2}{\sqrt{|\xi|^2 +m} +m} + \mu\right) |\hat{u}_c (\xi)|^2 d\xi
\\
&\geq C_{\mu,m} \int_{\mathbb{R}^n} (|\xi|+1) |\hat{u}_c (\xi)|^2 dx \geq C \left( \int_{\mathbb{R}^n} u_c^{p}(x) dx \right)^{\frac{2}{p}}.
\end{split}
\end{equation}
This yields that $\int_{\mathbb{R}^n} u_c^{p+1}(x) dx$ is lower bounded uniformly for $c>1$. Also, from \eqref{eq-up-1} we know that $\{u_c\}_{c>1}$ is bounded in $H^{1/2} (\mathbb{R}^n)$, and so $\{u_c\}_{c>1}$  is bounded in $L^{\frac{2n}{n-1}}(\mathbb{R}^n)$ by the Sobolev embedding. The lemma is proved.
\end{proof}
From the elliptic estimates (see for example, \cite{GT} and \cite{CN2}),
we see that each $u_c \in L^{q}(\mathbb{R}^n)$ for all $q \in [2,\infty]$.
Then the equation \eqref{main-eq} says that each solution $u_c$ is contained in $H^{1}(\mathbb{R}^n)$.
In the following theorem, we obtain a sharp uniform boundedness of $u_c$ in $H^1(\mathbb{R}^n)$ space with respect to $c>1$.
\begin{lem}\label{lem-uni-h1}
The ground state solutions $\{ u_c\}_{c>1}$ are uniformly bounded in $H^1 (\mathbb{R}^n)$.
More precisely, we have
\begin{equation}\label{eq-h1-1}
\limsup_{c \rightarrow \infty}\left( \|\nabla u_c\|_{L^2 (\mathbb{R}^n)}^2 + 2m \mu \| u_c\|_{L^2 (\mathbb{R}^n)}^2 \right) \leq
2m\limsup_{c \rightarrow \infty} \int_{\mathbb{R}^n} |u_c|^{p} dx.
\end{equation}
\end{lem}
\begin{proof}
Using \eqref{main-eq} we have
\begin{equation*}
\begin{split}
&c^2 \| \nabla u_c\|_{L^2 (\mathbb{R}^n)}^2 + m^2 c^4 \|u_c\|_{L^2 (\mathbb{R}^n)}^2
\\
&\qquad = \left\langle \sqrt{-c^2 \Delta +m^2 c^4} u_{c}, \quad \sqrt{-c^2 \Delta +m^2 c^4} u_{c} \right\rangle
\\
&\qquad = \left\langle mc^2 u_c - \mu u_{c} + u_c^{p-1},\quad mc^2 u_c - \mu u_{c} + u_c^{p-1}\right\rangle
\\
&\qquad = m^2 c^4 \| u_c\|_{L^2 (\mathbb{R}^n)}^2 - 2m c^2 \mu \|u_c\|_{L^2 (\mathbb{R}^n)}^2 + \mu^2 \| u_c\|_{L^2 (\mathbb{R}^n)}^2 + 2 (mc^2 - \mu) \int_{\mathbb{R}^n} u_c^{p} dx + \int_{\mathbb{R}^n} u_c^{2(p-1)} dx.
\end{split}
\end{equation*}
This leads to
\begin{equation}\label{eq-h1}
\begin{split}
&c^2 \| \nabla u_c\|_{L^2 (\mathbb{R}^n)}^2 + 2m c^2 \mu \|u_c\|_{L^2 (\mathbb{R}^n)}^2
\\
&\qquad\quad= \mu^2 \|u_c\|_{L^2 (\mathbb{R}^n)}^2 +  2(mc^2 -\mu) \int_{\mathbb{R}^n} u_c^{p} dx + \int_{\mathbb{R}^n}u_c^{2(p-1)} dx
\\
&\qquad\quad \leq (\mu^2 +1) \|u_c\|_{L^2 (\mathbb{R}^n)}^2 +  2(mc^2 -\mu) \int_{\mathbb{R}^n} u_c^{p} dx +  C \int_{\mathbb{R}^n}u_c^{\frac{2(n+1)}{n-1}} dx,
\end{split}
\end{equation}
where we applied Young's inequality $x^{2(p-1)}\leq x^2 + C x^{\frac{2(n+1)}{n-1}}$ for some $C = C(n) >0$. Next, we shall apply the interpolation inequality
\begin{equation}
\| u_c\|_{L^{\frac{2(n+1)}{n-1}}(\mathbb{R}^n)} \leq C  \|u_c\|_{L^{\frac{2n}{n-1}}(\mathbb{R}^n)}^{a} \|u_c\|_{L^{\frac{2n}{n-2}}(\mathbb{R}^n)}^{(1-a)},
\end{equation}
with a constant $C= C(n)>0$ and a value $a \in (0,1)$ such that
\begin{equation}\label{eq-h2}
\frac{a(n-1)}{2n} + \frac{(1-a) (n-2)}{2n} = \frac{(n-1)}{2(n+1)}.
\end{equation}
Applying this inequality to \eqref{eq-h1} we get
\begin{equation}
\begin{split}
&c^2 \| \nabla u_c\|_{L^2 (\mathbb{R}^n)}^2 + 2m c^2 \mu \|u_c\|_{L^2 (\mathbb{R}^n)}^2
\\
&\qquad\leq 2(mc^2 -\mu) \int_{\mathbb{R}^n} u_c^{p} dx   + \int_{\mathbb{R}^n} u_c^2 dx + \int_{\mathbb{R}^n} u_c^{\frac{2(n+1)}{n-1}} dx
\\
&\qquad \leq 2(mc^2 -\mu) \int_{\mathbb{R}^n} u_c^{p} dx   + \int_{\mathbb{R}^n} u_c^2 dx + C_n  \|u_c\|_{L^{\frac{2n}{n-1}}(\mathbb{R}^n)}^{2a(p-1)} \|u_c\|_{L^{\frac{2n}{n-2}}(\mathbb{R}^n)}^{2(1-a)(p-1)}
\\
&\qquad \leq 2(mc^2 -\mu) \int_{\mathbb{R}^n} u_c^{p} dx   + \int_{\mathbb{R}^n} u_c^2 dx + S_n  \|u_c\|_{L^{\frac{2n}{n-1}}(\mathbb{R}^n)}^{2a(p-1)} \|\nabla u_c\|_{L^2(\mathbb{R}^n)}^{2(1-a)(p-1)}.
\end{split}
\end{equation}
By Lemma \ref{lem-upc} there is a constant $A>1$ independent of $n$ such that
\begin{equation}
\max\left\{ \int_{\mathbb{R}^n} u_c^{p+1} dx,~ \int_{\mathbb{R}^n} u_c^2 dx, ~S_n \|u_c\|_{L^{\frac{2n}{n-1}}(\mathbb{R}^n)}^{2a(p-1)}\right\}\leq A.
\end{equation}
Now we have
\begin{equation}\label{eq-h3}
\|\nabla u_c\|_{L^2 (\mathbb{R}^n)}^2 + 2m \mu \| u_c\|_{L^2 (\mathbb{R}^n)}^2 \leq \left(2m-2\frac{\mu}{c^2}\right) \int_{\mathbb{R}^n} u_c^{p} dx  + \frac{A}{c^2} + \frac{A}{c^2} \| \nabla u_c\|_{L^2 (\mathbb{R}^n)}^{2(1-a)(p-1)}.
\end{equation}
The relation \eqref{eq-h2} leads that $a = \frac{2}{n+1}$ and so
\begin{equation}
2 (1-a)\cdot (p-1) < 2\, \frac{(n-1)}{n+1} \cdot \frac{(n+1)}{n-1} =2.
\end{equation}
Thus we can deduce from the above inequality \eqref{eq-h3} that
\begin{equation}
\sup_{c > 0}\|\nabla u_c\|_{L^2 (\mathbb{R}^n)} < \infty.
\end{equation}
Using this fact and taking the limit $c \rightarrow \infty$ in \eqref{eq-h3} we get
\begin{equation}
\limsup_{c \rightarrow \infty}\left( \|\nabla u_c\|_{L^2 (\mathbb{R}^n)}^2 + 2m \mu \| u_c\|_{L^2 (\mathbb{R}^n)}^2 \right)
\leq 2m\limsup_{c \rightarrow \infty} \int_{\mathbb{R}^n} u_c^{p} dx.
\end{equation}
The lemma is proved.
\end{proof}

\section{Nonrelativistic limit}
In this section we finish our paper with showing the nonrelativistic limit. We recall the result of Kwong \cite{Kw} that there exists a unique positive radial solution $u_{\infty} \in H^1 (\mathbb{R}^n)$ of the problem
\begin{equation}\label{eq-nr}
-\frac{1}{2m} \Delta U + \mu U = U^{p-1}.
\end{equation}
First we shall obtain a convergence result in a weak sense under a weak condition on the sequence of solutions.
\begin{lem}\label{lem-lvw} 
Let $\{ v_c\}_{c >1} \subset H^{1/2}(\mathbb{R}^n)$ be an one parameter family of weak solutions of
\begin{equation}\label{eq-uc}
c \left( \sqrt{-\Delta + m^2 c^2} - mc \right) v_c + \mu v_c = |v_c|^{p-2} v_c.
\end{equation}
Suppose that $\{ v_c\}$ is bounded in both of $L^{2}(\mathbb{R}^n)$ and  $L^{p}(\mathbb{R}^n)$ uniformly for $c$.
Let $v$ be a weak limit of $v_c$ in $L^{2}(\mathbb{R}^n)$ and $L^p(\mathbb{R}^n)$, up to a subsequence.
Then $v$ is a very weak solution of
\begin{equation}
-\frac{1}{2m}\Delta v + \mu v = |v|^{p-2} v
\end{equation}
in the sense that
\begin{equation}
\int_{\mathbb{R}^n} -\frac{1}{2m} v \Delta \phi (x) + \mu v \phi (x) dx = \int_{\mathbb{R}^n} |v|^{p-2} v \phi (x) dx
\end{equation}
for any $\phi \in \mathcal{S}(\mathbb{R}^n)$, where $\mathcal{S}(\mathbb{R}^n)$ denotes the Schwartz class.
\end{lem}
\begin{proof}
Since $v_c$ is a weak solution to \eqref{eq-uc},  for any $\phi \in \mathcal{S}(\mathbb{R}^n)$ we have
\begin{equation}\label{eq-uc-1}
\begin{split}
&\int_{\mathbb{R}^n} |v_c|^{p-2} v_c (x) \phi (x) dx
\\
&= \int_{\mathbb{R}^n} c (-\Delta + m^2 c^2)^{1/4} v_c (x) (-\Delta + m^2 c^2)^{1/4} \phi (x)  - mc^2  v_c  \phi (x) + \mu v_c  (x) \phi (x) dx
\\
& = \int_{\mathbb{R}^n} c v_c (x) (-\Delta + m^2 c^2)^{1/2} \phi (x)  - mc^2  v_c  \phi (x) + \mu v_c  (x) \phi (x) dx.
\end{split}
\end{equation}
First, since $\mathcal{S}(\mathbb{R}^n) \subset (L^q (\mathbb{R}^n))^{*}$ for any $q \in [1,\infty]$, we have
\begin{equation}\label{eq-uc-2}
\lim_{c \rightarrow \infty} \int_{\mathbb{R}^n} \mu v_c (x) \phi(x) dx =  \int_{\mathbb{R}^n} \mu v (x) \phi (x) dx
\end{equation}
and
\begin{equation}\label{eq-uc-3}
\lim_{c \rightarrow \infty} \int_{\mathbb{R}^n} |v_c|^{p-2} v_c (x) \phi (x) dx = \int_{\mathbb{R}^n} |v|^{p-2} v (x)\phi (x) dx
\end{equation}
from the weak convergence of $v_c$ and $|v_c|^{p-2}v_c$ to $v$ and $|v|^{p-2}v$
respectively.
Next we claim that for each $\phi \in \mathcal{S}(\mathbb{R}^n)$,
\begin{equation}\label{eq-uc-4}
c \left( \sqrt{-\Delta +m^2 c^2}-mc \right) \phi \rightarrow -\frac{\Delta}{2m}\phi
\end{equation}
strongly in $L^{2}(\mathbb{R}^n)$. To show this, we use the Plancherel theorem to have
\begin{equation}
\begin{split}
&\left\| c \left( \sqrt{-\Delta + m^2 c^2}-mc\right) \phi + \frac{\Delta}{2m} \phi \right\|_{L^2 (\mathbb{R}^n)}^2
\\
&=\int_{\mathbb{R}^n} \left| c \left(\sqrt{-\Delta +m^2 c^2} -mc\right) \phi + \frac{\Delta}{2m} \phi \right|^2 dx
\\
&= \int_{\mathbb{R}^n} \left|c\left(\sqrt{\xi^2 +m^2 c^2}-mc \right) - \frac{|\xi|^2}{2m}\right|^2 \widehat{\phi} (\xi)^2 d\xi
  = \int_{\mathbb{R}^n} \left| \frac{|\xi|^2}{\sqrt{\frac{|\xi|^2}{c^2} +m^2} + m} - \frac{|\xi|^2}{2m} \right|^2 \widehat{\phi} (\xi)^2 d\xi.
\end{split}
\end{equation}
We note that for any $c \geq 1$, it holds that
\begin{equation}
\left| \frac{|\xi|^2}{\sqrt{\frac{|\xi|^2}{c^2} +m^2} + m} - \frac{|\xi|^2}{2m} \right| \leq \left(\frac{|\xi|^2}{\sqrt{|\xi|^2 +m^2} + m} + \frac{|\xi|^2}{2m}\right),
\end{equation}
and
\begin{equation}
\int_{\mathbb{R}^n} \left(\frac{|\xi|^2}{\sqrt{|\xi|^2 +m^2} + m} + \frac{|\xi|^2}{2m}\right)^2 |\widehat{\phi}(\xi)|^2 d\xi  < \infty,
\end{equation}
since $\phi \in \mathcal{S}(\mathbb{R}^n)$. In addition, for each $\xi \in \mathbb{R}^n$ we trivially have
\begin{equation}
\lim_{c \rightarrow \infty} \left| \frac{|\xi|^2}{\sqrt{\frac{|\xi|^2}{c^2} +m^2} + m} - \frac{|\xi|^2}{2m} \right|^2 \widehat{\phi} (\xi)^2 =0.
\end{equation}
Thus, we may use the Lebesque dominated convergence theorem to conclude that
\begin{equation}
\lim_{c \rightarrow \infty} \int_{\mathbb{R}^n}\left| \frac{|\xi|^2}{\sqrt{\frac{|\xi|^2}{c^2} +m^2} + m} - \frac{|\xi|^2}{2m} \right|^2 \widehat{\phi} (\xi)^2 d\xi = 0.
\end{equation}
The claim is proved.

Combining \eqref{eq-uc-4} with the fact that $\{v_c\}_{c>1}$ is bounded in $L^{2}(\mathbb{R}^n)$ and $v_c \rightarrow v$ weakly in $L^{2}(\mathbb{R}^n)$, we deduce that
\begin{equation}\label{eq-uc-5}
\begin{split}
\lim_{c \rightarrow \infty} \int_{\mathbb{R}^n} v_c (x) c \left( \sqrt{-\Delta + m^2 c^2} -mc \right) \phi (x) dx &= \lim_{c \rightarrow \infty}\int_{\mathbb{R}^n} v_c \left( -\frac{1}{2m} \Delta \phi \right) (x) dx
\\
&= -\frac{1}{2m} \int_{\mathbb{R}^n} v \Delta \phi (x) dx.
\end{split}
\end{equation}
Now we see from \eqref{eq-uc-1}, \eqref{eq-uc-2}, \eqref{eq-uc-3} and \eqref{eq-uc-5} that
\begin{equation}
\int_{\mathbb{R}^n} -\frac{1}{2m} v \Delta \phi (x) +\mu v \phi(x) + |v|^{p-2} v (x) \phi (x) dx =0.
\end{equation}
Thus $v$ is a very weak solution of the problem
\begin{equation}
-\frac{1}{2m} \Delta v + \mu v + |v|^{p-2} v =0.
\end{equation}
The proof is complete.
\end{proof}
In the next lemma we shall prove a basic fact that a very weak solution is also a weak solution under a suitable assumption.
\begin{lem}\label{lem-lvw2}
Assume $v \in L^{p}(\mathbb{R}^n)$ is a very weak solution of
\begin{equation}\label{eq-leq}
-\frac{1}{2m} \Delta v + \mu v = |v|^{p-2}v.
\end{equation}
Then $v$ is contained in $H^1 (\mathbb{R}^n)$ and is a weak solution to problem \eqref{eq-leq}.
\end{lem}
\begin{proof}
Let  $w \in H^{1}(\mathbb{R}^n)$ be a unique weak solution of
\begin{equation}
-\frac{1}{2m} \Delta w + \mu w = |v|^{p-2}v.
\end{equation}
Observe $w$ is well defined as a function in $H^{1}(\mathbb{R}^n)$ because $|v|^{p-2}v \in L^{\frac{p}{p-1}}(\mathbb{R}^n)$ and $\frac{p}{p-1} > \frac{2n}{n+1}$ by the assumption that $p < \frac{2n}{n-1}$. 
Then  $w$ also satisfies
\begin{equation}
\int_{\mathbb{R}^n}  w (x)\left(-\frac{1}{2m}\Delta +\mu \right)\phi (x) dx = \int_{\mathbb{R}^n} |v|^{p-2}v(x) \phi (x) dx
\end{equation}
for every $\phi \in \mathcal{S}(\mathbb{R}^n)$ so that we have
\begin{equation}\label{eq-vw-2}
\int_{\mathbb{R}^n} (v-w) \left(-\frac{1}{2m}\Delta +\mu \right)\phi (x) dx = \int_{\mathbb{R}^n} |v|^{p-2}v(x)\phi (x) dx -\int_{\mathbb{R}^n} |v|^{p-2}v (x)\phi (x) dx=0.
\end{equation}
On the other hand, we claim that if $h \in L^{p} (\mathbb{R}^n)$ is a very weak solution of $(-\frac{1}{2m}\Delta +\mu) h =0$, i.e.,
\begin{equation}\label{eq-vw-1}
\int_{\mathbb{R}^n} h(x) \left(-\frac{1}{2m}\Delta +\mu \right) \phi (x) dx = 0,
\end{equation}
for any $\phi \in \mathcal{S} (\mathbb{R}^n)$, then $h =0$ holds. To show this, 
for any given function $\psi \in C_0^\infty(\mathbb{R}^n)$, we set $\phi (x) =(-\frac{1}{2m}\Delta+\mu)^{-1} \psi (x) \in \mathcal{S} (\mathbb{R}^n )$ so that $(-\frac{1}{2m}\Delta +\mu)\phi = \psi$. Then
\begin{equation}
\int_{\mathbb{R}^n} h(x) \psi (x) dx =0\quad \textrm{for all}~ \phi \in C_0^\infty(\mathbb{R}^n),
\end{equation}
which leads to that $h (x)=0$ holds almost everywhere.  From this uniqueness property, we see that $v =w \in H^1 (\mathbb{R}^n)$ by \eqref{eq-vw-2}, and hence $v$ is a weak solution of \eqref{eq-leq}.
\end{proof}
Now we shall prove the $H^1 (\mathbb{R}^n)$ convergence result of $\{u_c\}$ by combining the above weak limit property and the sharp $H^1 (\mathbb{R}^n)$ bound of Lemma \ref{lem-uni-h1}.

\begin{lem} The ground state solution $u_c$ converges to $u_{\infty}$ strongly in $H^1 (\mathbb{R}^n)$,
up to a subsequence, as $c \rightarrow \infty$.
\end{lem}
\begin{proof}
We recall that $\{u_c\}_{c>1}$ is bounded in $L^2 (\mathbb{R}^n) \cap L^{p}(\mathbb{R}^n)$ by Lemma \ref{lem-upc}. Thus, applying Lemma \ref{lem-lvw}, the solution $u_c$ converges to a function $v \in L^{p}(\mathbb{R}^n)$ weakly in $L^p (\mathbb{R}^n)$ and $v$ is a solution to
\begin{equation}
-\frac{1}{2m}\Delta v + \mu v = v^{p-1}.
\end{equation}
By Lemma \ref{lem-lvw2} the function $v$ is a weak solution.
From the compact Sobolev embedding (Lemma \ref{fractional-embedding}) and
lower $L^p$ boundedness (Lemma \ref{lem-upc}), we can see $v$ is nontrivial, nonnegative, radially symmetric, and thus $v= u_{\infty}$ by the uniqueness result of Kwong \cite{Kw}. We note from \eqref{eq-nr} that $u_{\infty}$ satisfies
\begin{equation}
\int_{\mathbb{R}^n} \nabla u_{\infty} \cdot \nabla u_{\infty} + 2m \mu u_{\infty}^2 dx= 2m \int_{\mathbb{R}^n} u_{\infty}^{p} dx.
\end{equation}
And, since $u_c$ converges to $u_{\infty}$ weakly in $H^1 (\mathbb{R}^n)$, we have
\begin{equation}
\begin{split}
&\lim_{c \rightarrow \infty} \int_{\mathbb{R}^n} \nabla u_c \nabla u_{\infty} + 2m \mu u_c u_{\infty} dx
\\
&\qquad\qquad = \int_{\mathbb{R}^n} \nabla u_{\infty} \cdot \nabla u_{\infty} + 2m \mu u_{\infty}^2 dx = 2m \int_{\mathbb{R}^n} u_{\infty}^{p} dx.
\end{split}
\end{equation}
Combining these two equalities with \eqref{eq-h1-1}, we get
\begin{equation}
\begin{split}
& \limsup_{c \rightarrow \infty}\left( \| \nabla (u_c - u_{\infty})\|_{L^2 (\mathbb{R}^n)}^2 + 2m \mu \|u_c - u_{\infty}\|_{L^2 (\mathbb{R}^n)}^2\right)
\\
&= \limsup_{c \rightarrow \infty} \biggl( \| \nabla u_c \|_{L^2 (\mathbb{R}^n)}^2 + 2m \mu \|u_c\|_{L^2 (\mathbb{R}^n)}^2 \biggr.
\\
&\qquad\qquad \biggl. -2 \left[ \int_{\mathbb{R}^n} \nabla u_c \nabla u_{\infty} + 2m \mu u_c u_{\infty} dx \right] + \| \nabla u_{\infty}\|_{L^2 (\mathbb{R}^n)}^2 + 2m \mu \| u_{\infty}\|_{L^2 (\mathbb{R}^n)}^2 \biggr)
\\
&\leq \limsup_{c \rightarrow \infty} \biggl( \| \nabla u_c \|_{L^2 (\mathbb{R}^n)}^2 + 2m \mu \|u_c\|_{L^2 (\mathbb{R}^n)}^2 \biggr)
\\
&\qquad+ \limsup_{c \rightarrow \infty} \left( -2 \left[ \int_{\mathbb{R}^n} \nabla u_c \nabla u_{\infty} + 2m \mu u_c u_{\infty} dx \right] + \| \nabla u_{\infty}\|_{L^2 (\mathbb{R}^n)}^2 + 2m \mu \| u_{\infty}\|_{L^2 (\mathbb{R}^n)}^2\right)
\\
&\leq  2m \int_{\mathbb{R}^n} u_{\infty}^{p} dx  -4m \int_{\mathbb{R}^n} u_{\infty}^{p} dx  +2m \int_{\mathbb{R}^n} u_{\infty}^{p} dx  = 0.
\end{split}
\end{equation}
 The lemma is proved.
\end{proof}

\medskip
\noindent \textbf{Acknowledgments}
The second author was supported by Basic Science Research Program through the National Research Foundation of Korea(NRF) funded by the Ministry of Education (NRF-2014R1A1A2054805).
Both authors were surpported by the POSCO TJ Park Science Fellowship.

\end{document}